\documentclass[11pt,oneside,english]{amsart}
\usepackage[T1]{fontenc}
\usepackage[latin9]{inputenc}
\usepackage{amsthm}
\usepackage{amssymb}

\makeatletter
\numberwithin{equation}{section}
\numberwithin{figure}{section}
\theoremstyle{plain}
\newtheorem{thm}{\protect\theoremname}
  \theoremstyle{definition}
  \newtheorem{problem}[thm]{\protect\problemname}

\makeatother

\usepackage{babel}
  \providecommand{\problemname}{Problem}
\providecommand{\theoremname}{Theorem}

\begin{document}

\title{A generalization of Schur's theorem and its application to consecutive
power residues}

\author{Carsten Dietzel}

\email{mat72129@stud.uni-stuttgart.de}
\begin{abstract}
This article provides a proof of a generalization of Schur's theorem
on the partition regularity of the equation $x+y=z$, which involves
a divisibility condition. This generalization will be utilized to
prove the existence of 'small' consecutive power residues modulo $p$,
where $p$ is a sufficiently large prime.
\end{abstract}
\maketitle

\section{A theorem of Brauer}

In \cite{key-1}, a theorem of Brauer is discussed. It says that,
given positive integers $k,m$, for each sufficiently large prime
$p$ there exist consecutive integers $r,r+1,...,r+m-1$, each of
which is a $k$th power residue modulo $p$, i.e. a $k$th power in
$\mathbb{Z}_{p}:=\mathbb{Z}/p\mathbb{Z}$. I will require a power
residue $r$ to fulfill the condition $0<r<p$, avoiding the trivial
cases $r=0,p$ for $k=2$.

Thus, for all but finitely many exceptional primes there exists a
minimal integer $r>0$ such that the condition of $r,r+1,...,r+m-1$
all being $k$th power residues holds. Define $r(k,m,p)$ to be that
integer.

Furthermore, define the function $\Lambda(k,m)$ to be the maximum
of $r(k,m,p)$ over all non-exceptional primes and set $\Lambda(k,m)=\infty$
if there is no such integer.

The authors of \cite{key-1} then ask if $\Lambda(k,2)$ is finite
for all $k$.

For $k=1$, this is a quite boring triviality and for $k=2$, an easy
exercise in elementary number theory.

Moreover, explicit values for $\Lambda(k,2)$, where $k=3,...,6$,
are given in \cite{key-1} 

The question was answered by Hildebrand in \cite{key-5,key-4} in
the affirmative:
\begin{thm}
$\Lambda(k,2)$ is finite for each $k$.
\end{thm}
But the original proof is quite complicated and depends on non-trivial
estimates of Dirichlet series.

The aim of this article is to provide a short combinatorial proof
of Theorem 1, offering a more elementary approach than the original
proof does.

Before proving Theorem 1, I'll prove Brauer's theorem for the special
case $m=2$, relying on a famous theorem of Schur on the partition
regularity of the equation $x+y=z$:
\begin{thm}
Let $\mathbb{Z}^{+}=C_{1}\uplus C_{2}\uplus...\uplus C_{l}$ be a
finite partition of the set of positive integers, then there exists
an $i\in\left\{ 1,2,...,l\right\} $ and $x,y,z\in C_{i}$ such that
$x+y=z$.

Moreover, there is a number $S(l)$ such that $x,y,z$ can be chosen
to be $\leq S(l)$.\end{thm}
\begin{proof}
See \cite{key-2}, p. 69.
\end{proof}
This already suffices to give a proof of Brauer's theorem for $m=2$.
This proof is nearly identical to Schur's original application of
Theorem 2 which he used to prove the solvability of the Fermat equation
in $\mathbb{Z}_{p}$ for large $p$. A proof can be found in \cite{key-2},
pp. 69-70 and will serve as a guideline for the proof of the following
theorem.
\begin{thm}
Let $k$ be fixed. Then, for each sufficiently large prime $p$ there
exist two consecutive $k$-th power residues modulo p.\end{thm}
\begin{proof}
Let $H\subseteq\mathbb{Z}_{p}^{\star}$ be the multiplicative subgroup
of $k$-th powers (except $0$) in $\mathbb{Z}_{p}^{\star}$. It's
clear that $[\mathbb{Z}_{p}^{\star}:H]\leq k$.

Now assume that $p>S(k)$ and let $H_{1},H_{2},...,H_{l}$ be the
cosets of $H$ in $\mathbb{Z}_{p}^{\star}$. The previous statement
says that $l\leq k$. The coset partition induces a partition of the
residue classes $\left\{ 1,2,...,S(k)\right\} $, so, by Theorem 2,
there is a class $H_{i}$ in which the equation $x+y=z$ can be solved.

Dividing the equation by $x$ leads to $yx^{-1}+1=zx^{-1}$. $y'=yx^{-1}$
and $z':=zx^{-1}$ must then be in $H$, and thus are consecutive
$k$th powers in $\mathbb{Z}_{p}$ fulfilling $y'+1=z'$.
\end{proof}
The idea of this proof will reappear when proving the finiteness of
$\Lambda(k,2)$. But in its 'classical' form, Theorem 2 doesn't allow
us to control the range of the residue classes of $y'$ and $z'$
which we got by dividing $y$ and $z$ by $x$ modulo $p$.

This disadvantage will be overcome by using a strengthening of Schur's
theorem which might be interesting on its own.

\section{A generalization of Schur's theorem}

By $R_{3}(k)$ we will mean the Ramsey number $R(\underset{k\, times}{\underbrace{3,3,...,3})}$,
i.e. the minimal number $R$ such that each edge partition of the
complete graph $K_{R}$ into $k$ parts has a triangle, all of whose
edges belong to the same partition class. $R_{3}(k)$, of course,
is finite, by Ramsey's theorem (see, for example, \cite{key-2}).

The proof of Theorem 4 will involve Ramsey's theorem in nearly the
same way as in the 'classical' proof of Schur's theorem. The construction
of the exactly right set of vertices will be crucial hereafter.
\begin{thm}
Let $\mathbb{Z}^{+}=C_{1}\uplus C_{2}\uplus...\uplus C_{l}$ be a
finite partition of the set of positive integers, then there exists
an $i\in\left\{ 1,2,...,l\right\} $ and $x,y,z\in C_{i}$ such that
$x+y=z$ and $x$ divides $y$.

Moreover, there is a number $S'(l)$ such that $x,y,z$ can be chosen
to be $\leq S'(l)$.\end{thm}
\begin{proof}
Firstly, construct an increasing sequence of positive integers recursively
as follows:

\begin{eqnarray*}
a_{1} & = & 1\\
a_{n+1} & = & \left(\sum_{i=1}^{n}a_{i}\right)!
\end{eqnarray*}

I know claim that for each triple $1\leq i<j<k$ there holds the following
divisibility relation:

\begin{eqnarray*}
\left(\sum_{m=i}^{j-1}a_{m}\right) & | & \left(\sum_{n=j}^{k-1}a_{n}\right)
\end{eqnarray*}

Obviously, it's enough to prove that

\begin{eqnarray*}
\left(\sum_{m=i}^{j-1}a_{m}\right) & | & a_{n}
\end{eqnarray*}

for $n\geq j$.

Because of $\sum_{m=i}^{j-1}a_{m}\leq\sum_{m=1}^{n-1}a_{m}$ and $a_{n}=\left(\sum_{m=1}^{n-1}a_{m}\right)!$,
the element $\sum_{m=i}^{j-1}a_{m}$ must be a factor in the product
expansion of $\left(\sum_{m=1}^{n-1}a_{m}\right)!$, so the divisibility
relation is proved.

Now let $R:=R_{3}(k)$ and label the vertices of the complete graph
$K_{R}$ by the numbers $1,2,...,R$.

Define an edge partition of this $K_{R}$ into classes $D_{1},D_{2},...,D_{l}$
as follows:

For $i<j$, let the edge $(i,j)$ belong to the class $D_{m}$ iff
$\left(\sum_{n=i}^{j-1}a_{n}\right)\in C_{m}$.

By Ramsey's theorem, a triangle $\left\{ i,j,k\right\} $ must exist,
all of whose edges belonging to the same class, say, $D_{M}$. Without
loss of generality, assume $i<j<k$ and set:

\begin{eqnarray*}
x & = & \sum_{n=i}^{j-1}a_{n}\\
y & = & \sum_{n=j}^{k-1}a_{n}\\
z & = & \sum_{n=i}^{k-1}a_{n}
\end{eqnarray*}

Clearly, $x+y=z$ and $x|y$, as we have just proved.
\end{proof}
If we tried to estimate the corresponding Schur numbers, the proof
of this theorem would give us quite astronomical upper bounds on $S'(l)$.

If one is a little more careful one might construct a sequence $b_{n}$
by setting $b_{1}=1$ and defining the following $b_{n}$ recursively
by

\begin{eqnarray*}
b_{n+1} & = & \prod_{1\leq i<j\leq n}\left(\sum_{k=i}^{j-1}b_{m}\right)
\end{eqnarray*}

and using these like the $a_{n}$ in the proof of Theorem 4. This
leads to better upper bounds, which are nevertheless certainly far
beyond the actual range of the numbers $S'(l)$.

One might also try to generalize Theorem 4 further in the style of
Theorem 3.1.2 in \cite{key-2}, in the following sense:
\begin{problem}
Is the following true:

Let $m$ be a positive integer and $\mathbb{Z}^{+}=C_{1}\uplus C_{2}\uplus...\uplus C_{l}$
be a finite partition of the set of positive integers, then there
exists an $i\in\left\{ 1,2,...,l\right\} $ and integers $x,y$ such
that $x,y,y+x,...,y+(m-1)x\in C_{i}$ and $x|y$.
\end{problem}
Without the divisibility condition, this is just Theorem 3.1.2. in
\cite{key-2} and imitating the proof of Theorem 3 immediately leads
to a full proof of Brauer's theorem mentioned at the beginning of
this article.

Why such a generalization cannot be true, though, will be explained
after the proof of Theorem 1 in the next section.

\section{Finiteness of $\Lambda(k,2)$}

Now we have all we need for proving Theorem 1:
\begin{proof}
Define $H\subseteq\mathbb{Z}_{p}^{\star}$ as in the proof of Theorem
3, but now take $p>S'(k)$ whose existence is given by Theorem 4.

The cosets $H_{1},H_{2},...,H_{l}$ now partition the set of residue
classes $\left\{ 1,2,...,S'(k)\right\} $. Again, we find a class
$H_{i}$ and $x,y,z\in H_{i}$ such that $x+y=z$ and $x|y$ respectively
$x|z$. This implies that $yx^{-1},zx^{-1}\in\left\{ 1,2,...,S'(k)\right\} $,
so, if we set $y'=yx^{-1}$, $z'=zx^{-1}$, we again get a pair $y',z'$
of consecutive $k$th power residues with the extra property that
$y',z'$ are bounded from above by $S'(k)$.

We have just proved that $\Lambda(k,2)\leq S'(k)$ and especially
that $\Lambda(k,2)$ is finite.
\end{proof}
This proof also shows why Problem 5 has to answered in the negative.
Else, even for $m=3$, the statement would imply in the same way as
in the proof of Theorem 1 the existence of a bound for the first three
consecutive $k$th power residues modulo $p$, which is independent
of $p$. But if $k=2$ is set, $\Lambda(2,3)$ would be finite which
contradicts the results in \cite{key-3} where $\Lambda(2,3)=\infty$
is shown.

At last, I will show that Theorem 4 also implies the following theorem
proved in \cite{key-4}.
\begin{thm}
For every $k$ there is a number $c_{0}=c_{0}(k)$, such that for
every multiplicative function $f:\mathbb{Z}^{+}\to\mathbb{C}$ (i.e.
$f(mn)=f(m)f(n)$) whose image is contained in the $k$'th roots of
unity, there is an $a\leq c_{0}$ with $f(a)=f(a+1)=1$.\end{thm}
\begin{proof}
Let $f$ be any such function. The preimages of the roots of unity
then give a partition of $\mathbb{Z}^{+}$ into at most $k$ parts.
So, there are $x,y,z\leq S'(k)$ with $f(x)=f(y)=f(z)$, fulfilling
$x+y=z$ and $x|y$.

Set $a=\frac{y}{x}\in\mathbb{Z}^{+}$. Then, dividing by $x$ gives
$a+1=\frac{z}{x}$.

The multiplicativity of $f$ shows that $f(a)=\frac{f(y)}{f(x)}=1$
and $f(a+1)=\frac{f(z)}{f(x)}=1$, as we wished.

Of course, $a\leq S'(k)=:c_{0}(k)$.\end{proof}

\end{document}